\documentclass[preprint, 11pt]{article}

\pdfoutput=1

\usepackage{amsmath,amsthm,amsfonts,amssymb,amscd}
\usepackage{graphicx,bm,color}
\usepackage{algorithm}
\usepackage{subcaption}
\usepackage[T1]{fontenc}
\usepackage{algpseudocode}
\usepackage{stmaryrd}
\usepackage{listings}
\usepackage{lineno}
\usepackage{mathptmx}
\usepackage{hyperref}
\usepackage{cleveref}
\usepackage[skins,theorems]{tcolorbox}
\usepackage{authblk}
\usepackage[latin1]{inputenc}
\usetikzlibrary{arrows,calc, decorations.markings, intersections}
\usepackage{caption}
\usepackage{cancel}
\usepackage{mathtools}
\usepackage{url}
\usepackage{ulem}
 
\def\keyword{\vspace{.5em}{\textbf{Keywords}.\,\relax}}

\definecolor{gray}{gray}{0.6}

\newtheorem{theorem}{Theorem}

\newtheorem{lemma}[theorem]{Lemma}

\theoremstyle{remark}

\newtheorem*{remark*}{Remark}

\newcommand{\vertiii}[1]{{\left\vert\kern-0.25ex\left\vert\kern-0.25ex\left\vert #1 
		\right\vert\kern-0.25ex\right\vert\kern-0.25ex\right\vert}}

\def\wta{\widetilde{a}}
\def\wtb{\widetilde{b}}
\def\wtd{\widetilde{d}}
\def\wtg{\widetilde{g}}
\def\wtm{\widetilde{m}}
\def\wtA{\widetilde{A}}
\def\wtB{\widetilde{B}}
\def\wtC{\widetilde{C}}
\def\wtD{\widetilde{D}}
\def\wtK{\widetilde{K}}
\def\wtM{\widetilde{M}}

\def\wtpi{\widetilde{\pi}}
\def\wttet{\widetilde{\theta}}

\def\bsf{\boldsymbol{f}}
\def\bsu{\boldsymbol{u}}
\def\bsx{\boldsymbol{x}}

\begin{document}

\title{Inverse properties of a class of seven-diagonal (near) Toeplitz matrices}
\author[1]{Bakytzhan Kurmanbek}
\author[2]{Yogi Erlangga}
\author[1, *]{Yerlan Amanbek}
\affil[1]{Nazarbayev University, Department of Mathematics, 53 Kabanbay Batyr Ave, Nur-Sultan 010000, Kazakhstan}
\affil[2]{Zayed University, Department of Mathematics, Abu Dhabi Campus, P.O. Box 144534, United Arab Emirates}
\affil[ ]{\textit{ bakytzhan.kurmanbek@nu.edu.kz, yogi.erlangga@zu.ac.ae, yerlan.amanbek@nu.edu.kz}}
\affil[*]{Corresponding author}
\date{\today}
\maketitle

\begin{abstract}
This paper discusses the explicit inverse of a class of seven-diagonal (near) Toeplitz matrices, which arises in the numerical solutions of nonlinear fourth-order differential equation with a finite difference method. A non-recurrence explicit inverse formula is derived using the Sherman-Morrison formula. Related to the fixed-point iteration used to solve the differential equation, we show the positivity of the inverse matrix and construct an upper bound for the norms of the inverse matrix, which can be used to predict the convergence of the method.
\end{abstract}

\keyword{seven-diagonal matrices, Toeplitz, exact inverse, upper bound of norm, nonlinear beam.}
\section{Introduction}

Many mathematical problems give rise to a system of equations that involves an inversion of a banded Toeplitz or near Toeplitz matrix. For example, a second-order or fourth-order finite difference approximation to a second-order differential operator results in a tridiagonal and, respectively, pentadiagonal Toeplitz matrix or a near Toeplitz matrix after the inclusion of boundary conditions. Inversions of this class of matrices have been studied extensively, and can be done very efficiently; see, e.g., ~\cite{Trench74, Mikkawy09R, Hadj08E,Kanal12BD,el2015new,tuanuasescu2019fast,el2008analytical}. In addition to the algorithmic development, many authors have contributed to the inverse properties of banded Toeplitz and near Toeplitz matrices, such as exact inverse formulas~\cite{hoskins1972some, dow2002explicit, Peluso01P,Lavis97S,Tan19}, bounds for entries of the inverse matrices, and bounds for the inverse norm. Examining formulas for determinant of such matrices can be also useful to explore the existence and uniqueness of solution related to the ordinary or partial differential problems \cite{ amanbek2020explicit,cinkir2012elementary,andjelic2020some,jia2019structure,kurmanbek2020proof}. 

An improved numerical accuracy can be attained via a higher-order approximation, but at the expense of increased bandwidth of the matrix in the system beyond five diagonals. This increased bandwidth not only increases the computational costs, but also complicates the analysis of the inverse properties. In many cases, the analysis demands for  additional conditions such as diagonal dominance or M-matrix~\cite{Meek83, Eijkhout88P, Lavis97S}. Exact inverse formulas, while can probably still be derived, may not be in an appealing form.

In this paper, we shall consider the inverse of $n \times n$ seven-diagonal near Toeplitz matrices associated with a fourth-order finite-difference discretization of the fourth-order differential operator $d^4/dx^4$:
\begin{equation} \label{form:An}
	A_n=
	\begin{pmatrix}
	a_0 & a_1 & 12 & -1 & 0 & \cdots & \cdots & \cdots  & 0 \\
	a_1 & 56 & -39 & 12 & -1 & \ddots & \ddots & \ddots  & \vdots \\
	12 & -39 & 56 & -39 & 12 & -1 & \ddots & \ddots  & \vdots \\
	-1 & 12 & -39 & 56 & -39 & 12 & -1 & \ddots  & \vdots \\
	0 & \ddots & \ddots & \ddots & \ddots & \ddots & \ddots & \ddots & 0\\
	\vdots & \ddots & \ddots &  \ddots &  \ddots & \ddots &  \ddots & \ddots & -1 \\
	\vdots & \ddots & \ddots & \ddots& \ddots& \ddots & \ddots  & -39  & 12 \\
	\vdots & \ddots & \ddots & \ddots & -1 & 12& -39 & 56 & a_1 \\
	0 & \cdots & \cdots & \cdots & 0 & -1 & 12 & a_1 & a_0 \\
	\end{pmatrix}
	_{n\times n}, \quad n \ge 7.
\end{equation}
The matrix~\eqref{form:An} is symmetric, centrosymmetric, nondiagonally dominant, and is not an M-matrix. The perturbation from the Toeplitz structure at the ''corner'' of the matrix can be caused by the inclusion of boundary conditions in the underlying boundary-value problems. 

An instance of application that involves~\eqref{form:An} is related to the nonlinear boundary-value problem
\begin{equation}
\displaystyle  EI \frac{d^4 u}{d x^4} = f(x, u), \quad x \in (0,1) \subset \mathbb{R} \label{eq:beam},
\end{equation}
with
\begin{equation}
  u(0) = u'(0) = u(1) = u'(1) = 0, \label{eq:beambc}
\end{equation}
Approximating the derivative by the fourth-order finite difference scheme results in the nonlinear system
\begin{equation} 
A_n \bsu = h^4 C_{EI}\bsf(\bsu),  \quad \bsu \in \mathbb{R}^n, \label{eq5}
\end{equation}
where $A_n$ is a near Toeplitz matrix of the form of~\eqref{form:An}, $h$ is the meshsize, and $C_{EI}$ is a physical constant.  Solution to the nonlinear system~\eqref{eq5} can be computed iteratively using a fixed-point method based on the iterands:
\begin{align}
  A_n \boldsymbol{u}^\ell = h^4 C_{EI} \boldsymbol{f}(\boldsymbol{u^{\ell-1}}), \quad \ell = 1,2,\dots,  \label{eq:fixedpoint}
\end{align}
for some initial solution vector $\bsu^0 \in \mathbb{R}^{n}$. For some class of the forcing term $f$, convergence of this method can be shown to depend on the $p$-norm of the inverse of $A_n$, $\|A^{-1}_n\|_p$, where $p \in \{1,2,\infty\}$.

In this paper, we derive an explicit, non-recurrence formula for the inverse of two special cases of the seven-diagonal matrix~\eqref{form:An}: (i) the Toeplitz case with with $a_0 = 56$ and $a_1 = -39$, and (ii) with $a_0 = 68$ and $a_1 = -40$, which corresponds to the boundary-value problem~\eqref{eq:beam} and~\eqref{eq:beambc}. The inverse formulas are then used to analyze some properties of the inverse matrices and to construct upper bounds for the norms of the inverses, in terms of the matrix size $n$ (which is linked to the meshsize $h$). While it is possible to construct a bound which is independent of $n$, an $n$-dependent bound is desirable as it can be used to  predict more accurately the convergence of the fixed-point method~\eqref{eq:fixedpoint} under mesh refinement.

In contrast with the tridiagonal and pentadiagonal cases, there does not exist a large body of results on the inverse of sevendiagonal (near) Toeplitz matrices. Literature on inverses of sevendiagonal matrices include~\cite{Jia10L, Lin14HJ, Lin16L} on algorithm development and~\cite{Huang09H} on inverse properties. While the class of matrices considered in this paper is quite narrow, our results are new, helpful in analyzing the convergence of the numerical recipe~\eqref{eq:fixedpoint}, and should contribute to the inverse theories of banded Toeplitz matrices.

The paper is organized as follows. After stating some preliminary results in Section~\ref{sec:prelim}, we derive the explicit formula for the Toeplitz matrix and an upper bound for the norms in Section~\ref{sec:Toeplitz}. Section~\ref{sec:Toeplitznear} is devoted to the formula for and norms of  the inverse of the near Toeplitz matrix. Some numerical results are presented in Section~\ref{sec:numer}, followed by concluding remarks in Section~\ref{sec:conclusion}.

%\sout{For the second-order finite-difference approximation to $d^4u/dx^4$, which leads to a pentadiagonal matrix, a sharp upper bound for the norm of the inverse can be derived by considering $\widetilde{A}_n$ as a low-rank perturbation to $A_n$. This approach however requires the knowledge of the inverse of $A_n$, which is lacking for its seven-diagonal counterpart.} {\bf We shall replace this by a paragraph on literature survey.}

\section{Preliminaries} \label{sec:prelim}

It is well known that a Toeplitz matrix cannot be decomposed into a product of two Toeplitz matrices. Some classes of Toeplitz matrices however admit a low-rank decomposition of the form
\begin{equation}
A_{n} = B_{n}C_{n} + \sigma UV^{T},  \quad \sigma \in \mathbb{R}, \label{rankdecom}
\end{equation}
where $B_n$ and $C_n$ are (near) Toeplitz, and $U$ and $V$ are $n \times m$ matrices, with $m < n$. Furthermore, if $A_n$ is nonsingular, the inverse matrix $A^{-1}_n$ can be computed using the Shermann-Morrison formula:
\begin{equation} \label{Aninv}
A^{-1}_{n} = D^{-1}_{n} - \sigma D^{-1}_{n}UM^{-1}_m V^{T}D^{-1}_{n},
\end{equation}
where $D_n = B_n C_n$, $M_m = I_m+\sigma V^{T}D^{-1}_{n}U \in \mathbb{R}^{m \times m}$, and $I_m$ is the identity matrix of size $m$.

As we shall see later, for the seven-diagonal matrices considered in this paper, the above low-rank decomposition involves the tridiagonal matrix
\begin{equation}
C_n =
\begin{pmatrix}
8 & -1 & 0 &  \cdots &  0 \\
-1 & 8 & \ddots  & \ddots&  \vdots \\
0 & \ddots & \ddots & \ddots  & 0 \\
\vdots & \ddots &\ddots & 8 &   -1 \\
0 & \cdots & 0 &  -1 &  8 
\end{pmatrix}_{n \times n}. \label{matCn}
\end{equation}
Some properties of $C_n$ are stated the following lemmas.

\begin{lemma} \label{lem:C}
$C_n$ is positive definite, with the inverse $C^{-1}_n = [c^{-1}_{i,j}]$, $i,j = 1,\dots,n$ given by
\begin{equation*}
c^{-1}_{i,j} = \frac{\gamma_{j}\gamma_{n+1-i}}{\gamma_{n+1}}
\end{equation*}
for $i \ge j$, and $c^{-1}_{i,j} = c^{-1}_{j,i}$, for $i<j$, where 
\begin{equation} \label{eq:defgamma}
\gamma_{k} = \left(r_1^{k}-r_2^{k}\right)/2\sqrt{15}, \quad k \in \mathbb{N},
\end{equation}
with $r_1 = 4+\sqrt{15}$ and $r_2 = 4-\sqrt{15}$.
\end{lemma}

\begin{proof}
The proof can be found in~\cite{dow2002explicit}.
\end{proof}

\begin{lemma}\label{lem:1}
Let $\gamma_k$ be defined as in (\ref{eq:defgamma}). Then the following holds for any $k \in \mathbb{N}$.
\begin{enumerate}
  \item[(i)] $\gamma_{k}+\gamma_{k-2}=8\gamma_{k-1}$;
  \item[(ii)] $4 \le  \displaystyle \frac{\gamma_{k+1}}{\gamma_k}  \le 8$.
\end{enumerate}
\end{lemma}
\begin{proof}
Set $r^{k}_{1} = (4 + \sqrt{15})^{k} = \alpha_{k}+\gamma_k\sqrt{15}$ and  $r^{k}_{2} = (4 - \sqrt{15})^{k} = \alpha_{k}-\gamma_k\sqrt{15}$. The parameters $\alpha_k$ and $\gamma_k$ satisfy the recurrence relations
	\begin{equation} %\label{equality}
	\renewcommand\arraystretch{1.4}
	\left\{
	\begin{array}{ll}
	\alpha_{k} = 4\alpha_{k-1}+15\gamma_{k-1},&  \\
	\gamma_{k} = \alpha_{k-1}+4\gamma_{k-1}. &  \\
	\end{array}
	\right.  \label{eq:alpgam}
	\end{equation}
Solving this linear equation system leads to the statement (i).

The (i) part implies $\gamma_{k} \le 8 \gamma_{k-1}$, which is the right inequality of the (ii) part. The left inequality of the (ii) part is proved by induction. For $k = 1$,  $\gamma_2/\gamma_1 = 4$. Thus, (ii) holds for $k = 1$. Suppose (ii) also holds for $k = j-1$, i.e., $\gamma_{j}/\gamma_{j-1} \ge 4$. For $k = j$, by utilizing the (i) part in the process,
$$
 \frac{\gamma_{j+1}}{\gamma_j} = \frac{\gamma_{j+1}}{\gamma_{j-1}} \frac{\gamma_{j-1}}{\gamma_j} = \left( 8 \frac{\gamma_j}{\gamma_{j-1}} - 1 \right) \frac{\gamma_{j-1}}{\gamma_j} = 8 - \frac{\gamma_{j-1}}{\gamma_j} \ge 4.
$$
\end{proof}

\begin{lemma} \label{lem:2}
	Let $\gamma_k$ be defined as in (\ref{eq:defgamma}). Then, for any $p \in \mathbb{N}$,
	\begin{equation} \label{eq:lem21}
	\sum_{k=1}^{p}\gamma_{k} = \frac{1}{6}(\gamma_{p+1}-\gamma_{p}-1),
	\end{equation}
	\begin{equation} \label{eq:lem22}
	\sum_{k=1}^{p}k\gamma_{k} = 
	\frac{1}{6}(p\gamma_{p+1}-(p+1)\gamma_{p}),
	\end{equation}
	\begin{equation} \label{eq:lem23}
	\sum_{k=1}^{p}k^2\gamma_{k} = 
	\frac{1}{18}((3p^{2}+1)\gamma_{p+1}-(3p^{2}+6p+4)\gamma_{p}-1)\;,\;\;
	\end{equation}
	\begin{equation} \label{eq:lem24}
    \sum_{k=1}^{p}k^3\gamma_{k} =
	\frac{1}{6}((p^{3}+p)\gamma_{p+1}-(p^{3}+3p^{2}+4p+2)\gamma_{p}).
	\end{equation}
\end{lemma}
\begin{proof}
The proof uses Vieta's formula and Lemma~\ref{lem:1}. Since the proof for each relation is similar, we shall show the proof only for 	
\eqref{eq:lem21} and \eqref{eq:lem22}.
	\begin{align*}
	\sum_{k=1}^{p}\gamma_{k} &= \sum_{k=1}^{p}\frac{r^{k}_{1}-r^{k}_{2}}{2\sqrt{15}} = \frac{1}{2\sqrt{15}}\left(\sum_{k=0}^{p}r^{k}_{1} - \sum_{k=0}^{p}r^{k}_{2}\right) = \frac{1}{2\sqrt{15}}\left(\frac{r^{p+1}_{1}-1}{r_{1}-1}-\frac{r^{p+1}_{2}-1}{r_{2}-1}\right)  \\
	&= \frac{1}{12\sqrt{15}}(r^{p+1}_{1}-r^{p}_{1}+r^{p}_{2}-r^{p+1}_{2}+r_{2}-r_{1}) = \frac{1}{6}(\gamma_{p+1}-\gamma_{p}-1).
	\end{align*}
Next,
	\begin{align*}
	\sum_{k=1}^{p}k\gamma_{k} &= 
	\sum_{k=1}^{p}\frac{k(r^{k}_{1}-r^{k}_{2})}{2\sqrt{15}} =
	\frac{1}{2\sqrt{15}}\left(r_{1}\sum_{k=1}^{p}k r^{k-1}_{1} -
	r_{2}\sum_{k=1}^{p}k r^{k-1}_{2}\right)  \\
	&=\frac{1}{2\sqrt{15}}\left(r_{1}\left(\sum_{k=1}^{p}r^{k}_{1}\right)^{'}-
	r_{2}\left(\sum_{k=1}^{p}r^{k}_{2}\right)^{'}\right)  \\
    &=\frac{1}{2\sqrt{15}}\left(\frac{pr^{p+2}_{1}-(p+1)r^{p+1}_{1}+r_{1}}{(r_{1}-1)^{2}} - 
	\frac{pr^{p+2}_{2}-(p+1)r^{p+1}_{2}+r_{2}}{(r_{2}-1)^{2}}\right) \\
	&=\frac{1}{36}\left(p\gamma_{p+2}-(3p+1)\gamma_{p+1}+(3p+2)\gamma_{p}-(p+1)\gamma_{p-1}\right).
	\end{align*}
In the above derivation, we have used the relation (\ref{eq:lem21}) to evaluate derivatives of the sum. The relation~\eqref{eq:lem22} is obtained from the above equation by applying Lemma \ref{lem:1} multiple times.
\end{proof}

\begin{lemma} \label{lem:3}
Let $\gamma_k$ be defined as in (\ref{eq:defgamma}) and $\alpha_{k}=(4 + \sqrt{15})^{k}-\gamma_k\sqrt{15}$. Then
	\begin{equation*}
	\alpha_{k}-\alpha_{k-2} = 30\gamma_{k-1}
	\end{equation*}
	is true for any $k \in \mathbb{N}$.
\end{lemma}

\begin{proof}
By using the recurrence relations $\gamma_{k}=\alpha_{k-1}+4\gamma_{k-1}$ and  $\alpha_{k} = 4\gamma_{k} - \gamma_{k-1}$ from~\eqref{eq:alpgam} we obtain $$\gamma_{k}-\gamma_{k-2}=2\alpha_{k-1}$$
Applying this identity to the $(k-2)$ and $(k-1)$ term, we have
	\begin{equation*}
	2(\alpha_{k}-\alpha_{k-2})=(\gamma_{k+1}-\gamma_{k-1})-(\gamma_{k-1}-\gamma_{k-3}) = \gamma_{k+1}-2\gamma_{k-1}+\gamma_{k-3}.
	\end{equation*}
	Applying Lemma~\ref{lem:1} several times to the above equation leads to the statement in the lemma.
\end{proof}

\begin{lemma} \label{lem:invCn}
For $p \in \{0,1,\infty\}$, $\|C^{-1}_n\|_p \le 1/6.$
\end{lemma}
\begin{proof}
A proof for general diagonally-dominant symmetric tridiagonal matrices $T_n = tril(-1,b,-1)$ is given in~\cite{amanbek01EK}; see also \cite{Tan19} for alternative bounds. 
For $b > 2$, 
\begin{align}
    \|T_n\|_{\infty} = \frac{1}{b-2} - \frac{2}{r_b^{\frac{n+1}{2}}},
\end{align}
where $r_b = \frac{1}{2}(b + \sqrt{b^2-4})$. For $C_n$, setting $b = 8$ leads to the bound in the lemma.
\end{proof}

\section{The Toeplitz case} \label{sec:Toeplitz}

In this section, we consider the case where~\eqref{form:An} is Toeplitz ($a_0 = 56$ and $a_1 = -39$). The seven-diagonal matrix $A_n$ can be decomposed into a rank-2 decomposition~\eqref{rankdecom}, with $\sigma = 1$,
\begin{equation}
\label{banded}
B_n=
\begin{pmatrix}
6 & -4 & 1 & 0 &  \cdots  & \cdots  & 0 \\
-4 & 6 & -4 & 1 & \ddots & \ddots   & \vdots  \\
1 & -4 & 6 & -4 & 1 & \ddots  &  \vdots  \\
0 & \ddots  & \ddots  & \ddots & \ddots & \ddots & 0 \\
 \vdots  & \ddots  & 1 & -4 & 6 & -4 & 1 \\
 \vdots  & \ddots  & \ddots  & 1 & -4 & 6 & -4 \\
0 & \cdots & \cdots & 0 & 1 & -4 & 6 
\end{pmatrix}
_{n\times n},
U =
\begin{pmatrix}
4 & 0 \\
-1 & 0 \\
0  & 0 \\
\vdots & \vdots \\
0 & 0 \\
0 & -1 \\
0 & 4 \\
\end{pmatrix}
_{n\times 2}, 
V =
\begin{pmatrix}
1 & 0 \\
0 & 0 \\
0  & 0 \\
\vdots & \vdots \\
0 & 0 \\
0 & 0 \\
0 & 1
\end{pmatrix}
_{n\times 2}, 
\end{equation}
and with $C_n$ given in~\eqref{matCn}. 
The matrix $B_n$ in the decomposition is nonsingular. The entries of the inverse matrix $B^{-1}_n$ is derived in~\cite{hoskins1972some} and given for $i \ge j$ by the formula
\begin{equation*}
b^{-1}_{i,j} = -\frac{(n+1-i)(n+2-i)j(j+1)}{6(n+1)(n+2)(n+3)} \left[(i+1)(j-1)(n+3) - i(j+2)(n+1)\right].
\end{equation*}

% \textcolor{blue}{Regarding the matrix $B_n$, we have the following lemma}
Regarding the matrix $B_n$, we have the following lemma
\begin{lemma} \label{lemma:spdB}
The matrix $B_n$ in~\eqref{banded} is positive definite.
\end{lemma}
\begin{proof}
For any vector $\bsx = (x_1 \, \dots \, x_n)^T \in \mathbb{R}^n$, 
\begin{align}
\bsx^{T}B_{n}\bsx &= 6(x_{1}^2 + x_{2}^2+\cdots+x_{n}^2) - 8(x_{1}x_{2}+x_{2}x_{3}+\cdots+x_{n}x_{n-1})+2(x_{1}x_{3}+x_{2}x_{4}+x_{3}x_{5}+\cdots+x_{n-2}x_{n}) \notag \\
&= x_{1}^2+(2x_{1}-x_{2})^2 + (x_{1}-2x_{2}+x_{3})^{2}+ \dots +(x_{n-2}-2x_{n-1}+x_{n})^2+(2x_{n}-x_{n-1})^2+x_{n}^2 \notag \\
&\ge 0, \notag
\end{align}
with equality holding only when $\bsx = 0$.
\end{proof}

\subsection{Exact inverse formula}

An explicit formula for $A_n^{-1}$ can be derived by evaluating the right-hand side of the decomposition~\eqref{Aninv}.

Starting with the first term on the right-hand side, let $D_n^{-1} := C_n^{-1} B_n^{-1} = [d^{-1}_{i,j}]$, with
\begin{align}
d^{-1}_{i,j} &= \sum_{k=1}^{n} c^{-1}_{i,k} b^{-1}_{k,j} \notag \\
        &= \sum_{k=1}^{j}c^{-1}_{i, k} b^{-1}_{j,k} + \sum_{k=j+1}^{i}c^{-1}_{i, k} b^{-1}_{k, j} +\sum_{k=i+1}^{n}c^{-1}_{k, i} b^{-1}_{k, j} \notag \\
&= \sum_{k=1}^{j}\frac{\gamma_{k}\gamma_{n+1-i}}{\gamma_{n+1}}b^{-1}_{j, k}+
\sum_{k=j+1}^{i}\frac{\gamma_{k}\gamma_{n+1-i}}{\gamma_{n+1}}b^{-1}_{k, j} +
\sum_{k =1}^{n-i}\frac{\gamma_{i}\gamma_{k}}{\gamma_{n+1}}b^{-1}_{n+1-k, j}, \label{eq:dinv}
\end{align}
due to symmetry of $B_n$ and $C_n$. Furthermore, $D^{-1}_n$ is centrosymmetric, due to the centrosymmetry of $B_n^{-1}$ and $C_n^{-1}$. By using Lemmas \ref{lem:1}--\ref{lem:3} and after necessary simplifications, we obtain, for $i \geq j$, 
\begin{align}\label{eq:Dinv}
	d^{-1}_{i,j} &= \frac{\gamma_j \gamma_{n+1-i}}{36 \gamma_{n+1}} + \eta \zeta_1\left(  \frac{\gamma_{n+1-i}\gamma_{i+1} - \gamma_{n-i} \gamma_i}{\gamma_{n+1}} \right) + \eta \left( \zeta_2 \frac{\gamma_{n+1-i}}{\gamma_{n+1}} + \zeta_3 \frac{\gamma_i}{\gamma_{n+1}} \right),
\end{align}
where
\begin{align*}
  \eta      &= - \frac{1}{6(n+1)(n+2)(n+3)}, \\
   \zeta_1 &= \frac{1}{6}j(j+1)\{ (j-3i-1)n^3 + (6j + 6i^2 - 12i -4)n^2 + ((-3i^2 - 3i +10)j - 3i^3 + 18i^2 - 15i -5)n \\
               & \quad \quad + (2i^3 - 3i^2 - 3i +5)j - 5i^3 +12i^2 - 6i - 2\},\\
   \zeta_2 &= \frac{1}{6}(n+1)j (n+1-j)(n+2-j), \\
   \zeta_3 &= \frac{1}{6}(n+1)j(j+1)(n+1-j).
\end{align*}

For the second term on the right-hand side of~\eqref{Aninv}, let $M_2 = [m_{i,j}] \in \mathbb{R}^{2 \times 2}$. From a direct calculation and the use of  Lemmas \ref{lem:1}--\ref{lem:3} and~\eqref{eq:Dinv}, one can show that 
\begin{align} 
m_{11} &= m_{22} = 1+ 4d^{-1}_{1,1}-d^{-1}_{1,2} = 1 + \frac{11n^{2}+5n}{36(n+1)(n+2)}+\frac{1}{36\gamma_{n+1}}\left(\alpha_{n}-\frac{2+2n\gamma_{n}}{n+2}\right), \label{eq:a} \\
m_{12} &= m_{21} = 4d^{-1}_{n,1}-d^{-1}_{n,2} = \frac{7n+4}{18(n+1)(n+2)}-\frac{1}{18\gamma_{n+1}}\left(2+\frac{n+\gamma_{n}}{n+2}\right). \label{eq:b}
\end{align}

\begin{lemma} \label{lem:M}
$\det(M)>0$.
\end{lemma}

\begin{proof}
Note that $11n^2-9n-8>9(n^2-n-1)>0$ and $36(n+2)\gamma_{n+1} - (2n-2)\gamma_{n}>(2n-2)(\gamma_{n+1}-\gamma_{n})>0$. Furthermore, $\gamma_{n+1}>\gamma_{n}$ and $\gamma_{n+1}>n$. Using these inequalities,  for $n \ge 2$,
\begin{align}
m_{11} - m_{12} &= 1 + \frac{11n^2-9n-8}{36(n+1)(n+2)} + \frac{1}{18 \gamma_{n+1}}\Big( \frac{\alpha_{n}}{2}+2-\frac{(n-1)(\gamma_{n}-1)}{n+2}\Big) \notag \\
&= \frac{11n^2-9n-8}{36(n+1)(n+2)} + \frac{\alpha_{n}(n+2)+6n+6+36(n+2)\gamma_{n+1} - (2n-2)\gamma_{n}}{36(n+2)\gamma_{n+1}} > 0. \notag
\end{align}

Next, for $n \geq 2$, using~\eqref{eq:b},
\begin{align}
m_{12} %&=\frac{7n+4}{18(n+1)(n+2)}-\frac{1}{18\gamma_{n+1}}\left(2+\frac{n+\gamma_{n}}{n+2}\right) %\notag \\
%&= \frac{(7n+4)\gamma_{n+1}-2(n+1)(n+2)-(n+1)(n+\gamma_{n})}{18(n+1)(n+2)\gamma_{n+1}} \notag \\
&= \frac{(n+1)(\gamma_{n+1}-\gamma_{n})+(6n+3)\gamma_{n+1}-3n^2-7n-4}{18(n+1)(n+2)\gamma_{n+1}} \notag \\
&> \frac{(n+1)(\gamma_{n+1}-\gamma_{n}) + (n-2)(3n+2)}{18(n+1)(n+2)\gamma_{n+1}} > 0. \notag 
\end{align}
Therefore, $m_{11} > m_{12} > 0$, and hence $m_{11}^2 - m_{12}^2 > 0$.
\end{proof}

%\begin{equation*}
%(I_2 + V^{T}D^{-1}_{n}U)^{-1}=\frac{1}{m_{11}^{2}-m_{12}^{2}}\begin{pmatrix}
%m_{11} & -m_{12} \\
%-m_{12} & m_{11} 
%\end{pmatrix},
%\end{equation*}
By combining the above results and simplifying the expressions for $i \geq j$, the explicit formula of the inverse of $A_n$ can be written as
\begin{align} 
a^{-1}_{i, j} = d^{-1}_{i,j} + \frac{1}{m_{11}^{2}-m_{12}^{2}}\left((m_{12}d^{-1}_{1, j} - m_{11}d^{-1}_{n, j})(4d^{-1}_{i,n}-d^{-1}_{i,n-1})+(m_{12}d^{-1}_{n, j}-m_{11}d^{-1}_{1, j})(4d^{-1}_{i,1}-d^{-1}_{i,2})\right). \label{eq:aijinv}
\end{align} 
%\begin{align}
%\begin{split}
%a^{-1}_{i, j}&= d^{-1}_{i, j}-\frac{m_{11}}{m_{11}^{2}-m_{12}^{2}}\left(d^{-1}_{1,j}(4d^{-1}_{i,1}-d^{-1}_{i,2})+d^{-1}_{n, j}(4d^{-1}_{i,n}-d^{-1}_{i,n-1})\right)+ \\
%&+\frac{m_{12}}{m_{11}^{2}-m_{12}^{2}}\left(d^{-1}_{1,j}(4d^{-1}_{i,n}-d^{-1}_{i,n-1})+d^{-1}_{n, j}(4d^{-1}_{i,1}-d^{-1}_{i,2})\right).
%\end{split}
%\end{align}
where the coefficents on the right-hand side are computed using (\ref{eq:defgamma}), (\ref{eq:Dinv}), (\ref{eq:a}) and (\ref{eq:b}).

\begin{theorem} \label{Ainvspd}
$A_n$ is positive definite.
\end{theorem}
\begin{proof}
Since $A_n$ is symmetric, the proof is based on Sylvester's criterion. In this case, we need to show that all upper left $k\times k$ submatrices of $A_n$ have positive determinant, $k = 1,\dots, n$. For $k = 1,\dots, 6$, the determinant can be shown to be positive via numerical calculation. For $k \ge 7$, since the submatrices retain the structure of $A_n$, we only need to show that $A_n$ has positive determinant. Using the generalized matrix determinant lemma, with $D_n = B_n C_n$,
\begin{align}
  \det (A_n) = \det (D_n + UV^T) = \det (I_2 + V^T D^{-1} U) \det (D_n) = \det (M) \det (B_n) \det (C_n). \notag
\end{align}
Positive definiteness of $B_n$ and $C_n$ (Lemmas~\ref{lemma:spdB} and~\ref{lem:C}, respectively) implies $\det (B_n) > 0$ and $\det(C_n) > 0$. Together with Lemma~\ref{lem:M}, we have $\det(A_n) > 0$, which proves the theorem.
\end{proof}

%\begin{corollary}
%$A^{-1}_n > 0$.
%\end{corollary}

\subsection{Bound of norms of inverse}

In this section we derive a bound of $\| A_n^{-1}\|_{p}$, for $p=1, 2, \infty$. The result is summarized in Theorem~\ref{theo:1} below:

%From the symmetricity property of the inverse matrix we conclude $$\|A^{-1}\|_{1} = \|A^{-1}\|_{\infty}$$
%Also, we have the following inequality: 
%$$\|A^{-1}\|_{2}\leq \sqrt{\|A^{-1}\|_{1}\|A^{-1}\|_{\infty}} =\|A^{-1}\|_{\infty} $$

\begin{theorem} \label{theo:1}
	Let $A_n$ be given as in (\ref{form:An}), with $a_0=56$ and $a_1 = -39$. Then the following inequality holds for $p \in \{1,2, \infty\}$:
	\begin{equation*}\label{thm:upperbound}
	 \|A^{-1}_n\|_{p} %\leq\frac{(n+1)^{2}(n+3)^{2}}{2221}+\frac{(55n^{2}+7n+14)(3n^{2}+18n+23)}{248832(n+2)} 
	 \le \frac{(n+1)^{2}(n+3)^{2}}{2304} + \frac{(n+1)^2}{432} + \frac{n+4}{24}.
	\end{equation*}	
%	\begin{equation*}\label{thm:upperbound}
%	 \|A^{-1}_n\|_{p} %\leq\frac{(n+1)^{2}(n+3)^{2}}{2221}+\frac{(55n^{2}+7n+14)(3n^{2}+18n+23)}{248832(n+2)} 
%	 \le \frac{(n+1)^{2}(n+3)^{2}}{2221} + \frac{(n+1)^2}{432} + \frac{n+4}{24}.
%	\end{equation*}	
\end{theorem}
 
%From previous paper about a seven-diagonal matrix inverse, we have the following results:
%\begin{equation}
%c^{-1}_{n}(i, j) = \frac{\gamma_{j}\gamma_{n+1-i}}{\gamma_{n+1}} \;\;\; for\;\; i\geq j
%\end{equation}
%\begin{equation}
%d^{-1}_{n}(i, j) = \frac{c^{-1}_{n}(i, j)}{36}-\frac{1}{36(n+1)(n+2)(n+3)}\left(f_{1}+f_{2}\right) \;\;\; for\;\; i\geq j
%\end{equation} where
%\begin{equation}
%f_{1} = \frac{(n+1-j)j(n+1)}{\gamma_{n+1}}(\gamma_{n+1-i}(n+2-j) +\gamma_{i}(j+1))
%\end{equation}
%\begin{multline}
%f_{2} = j\left(j+1\right)(2(j-1)i^3-3i(n+1)(i^2+ij+j+(n+2)(n+1-2i))+\\+(n+1)\left((n+1)(n+2)(j-1)+j(2n+3)\right))
%\end{multline}
%\begin{equation}
%a^{-1}_{i, j}= d^{-1}_{i, j}-\frac{a}{a^{2}-b^{2}}\left(d^{-1}_{1,j}(4d^{-1}_{i,1}-d^{-1}_{i,2})+d^{-1}_{n, j}(4d^{-1}_{i,n}-d^{-1}_{i,n-1})\right)+
%\end{equation}
%\begin{equation*}
%+\frac{b}{a^{2}-b^{2}}\left(d^{-1}_{1,j}(4d^{-1}_{i,n}-d^{-1}_{i,n-1})+d^{-1}_{n, j}(4d^{-1}_{i,1}-d^{-1}_{i,2})\right) \;\;\; for\;\; i\geq j
%\end{equation*} 
%\begin{equation}
%a=1+\frac{11n^{2}+5n}{36(n+1)(n+2)}+\frac{1}{36\gamma_{n+1}}\left(\alpha_{n}-\frac{2+2n\gamma_{n}}{n+2}\right)
%\end{equation}
%\begin{equation}
%b=\frac{7n+4}{18(n+1)(n+2)}-\frac{1}{18\gamma_{n+1}}\left(2+\frac{n+\gamma_{n}}{n+2}\right)
%\end{equation}
%So, 
\begin{proof}
By the symmetry of $A_n$, $\|A^{-1}\|_{1} = \|A^{-1}\|_{\infty}$ and  $\|A^{-1}\|_{2}\leq \sqrt{\|A^{-1}\|_{1}\|A^{-1}\|_{\infty}} =\|A^{-1}\|_{\infty} $. Therefore,  it suffices to prove the result for $p=\infty$.  
	
%To proceed, we rewrite the explicit formula for $A^{-1}_n$ into the following form:
%\begin{equation*}
%a^{-1}_{i, j}= d^{-1}_{i, j}+\frac{1}{m_{11}^{2}-m_{12}^{2}}\left((m_{12}d^{-1}_{1, j} - m_{11}d^{-1}_{n, j})(4d^{-1}_{i,n}-d^{-1}_{i,n-1})+(m_{12}d^{-1}_{n, j}-m_{11}d^{-1}_{1, j})(4d^{-1}_{i,1}-d^{-1}_{i,2})\right).
%\end{equation*}
The positive definiteness of $A_n$ (Theorem~\ref{Ainvspd}) implies that $A_n^{-1}>0$. With
\begin{align}
m_{12}d^{-1}_{1, j} - m_{11}d^{-1}_{n, j} &= m_{12}d^{-1}_{n, n+1-j} - m_{11}d^{-1}_{n, j}, \notag \\
m_{12}d^{-1}_{n, j}-m_{11}d^{-1}_{1, j} &= m_{12}d^{-1}_{n, j}-m_{11}d^{-1}_{n, n+1-j}, \notag
\end{align}
due to centrosymmetry of $D_n$ and $M$, for the $i$-th rowsum of $A^{-1}_n$, we have
\begin{align*}
\sum^{n}_{j = 1} |a^{-1}_{i, j}| &= \sum^{n}_{j = 1}d^{-1}_{i, j} \notag \\
&+ \frac{1}{m_{11}^{2}-m_{12}^{2}}\left((4d^{-1}_{i,n}-d^{-1}_{i,n-1})\sum^{n}_{j = 1}(m_{12}d^{-1}_{1, j} - m_{11}d^{-1}_{n, j})+(4d^{-1}_{i,1}-d^{-1}_{i,2})\sum^{n}_{j = 1}(m_{12}d^{-1}_{n, j}-m_{11}d^{-1}_{1,j})\right) \\
%&= \sum^{n}_{j = 1}d^{-1}_{i, j} + \frac{1}{m_{11}^{2}-m_{12}^{2}}\left((4d^{-1}_{i,n}-d^{-1}_{i,n-1})(m_{12}-m_{11})\sum^{n}_{j = 1}d^{-1}_{n, j}+(4d^{-1}_{i,1}-d^{-1}_{i,2})(m_{12}-m_{11})\sum^{n}_{j = 1}d^{-1}_{n, j}\right) \\
&= \sum^{n}_{j = 1}d^{-1}_{i, j} - \frac{\sum^{n}_{j = 1}d^{-1}_{n, j}}{m_{11}+m_{12}}\left(4(d^{-1}_{i,n}+d^{-1}_{i,1})-(d^{-1}_{i,n-1}+d^{-1}_{i,2})\right).
\end{align*}
%Denote $$K(i) = 4 \underbrace{(d^{-1}_{i,n}+d^{-1}_{i,1})}_{\mathbb{T}_1}-(d^{-1}_{i,n-1}+d^{-1}_{i,2})$$
Then
\begin{equation}
\| A^{-1}_n\|_{\infty} = \max_{1\leq i \leq n}\sum^{n}_{j = 1} |a^{-1}_{i, j}| \leq \underbrace{\max_{1\leq i \leq n}\sum^{n}_{j = 1}d^{-1}_{i, j}}_{\pi_1} + \frac{1}{m_{11}+m_{12}}\underbrace{(\sum^{n}_{j = 1}d^{-1}_{n, j})}_{\pi_2} \underbrace{\max_{1\leq i \leq n} g(i)}_{\pi_3}
\end{equation}
where 
\begin{equation*}
g(i) = 4 \underbrace{ (d^{-1}_{i,n}+d^{-1}_{i,1})}_{\theta_1}- \underbrace{(d^{-1}_{i,n-1}+d^{-1}_{i,2})}_{\theta_2}.
\end{equation*}

First of all, with $d^{-1}_{i, n} = d^{-1}_{n+1-i, 1}$, %due to centrosymmetry,
%\subsection{Bound for $K(i)$}
\begin{align}
\theta_1 &= d^{-1}_{n+1-i, 1}+d^{-1}_{i, 1} = \frac{\gamma_{1}(\gamma_{i}+\gamma_{n+1-i})}{36\gamma_{n+1}} - \frac{1}{36(n+2)}\left(\frac{n(\gamma_{i}+\gamma_{n+1-i})}{\gamma_{n+1}}+2(1-3i(n+1-i))\right), \notag \\
%\frac{\gamma_{1}(\gamma_{i}+\gamma_{n+1-i})}{36\gamma_{n+1}} - \frac{1}{36(n+1)(n+2)(n+3)}\left(f_{1}+f_{2}\right)
%where $$f_{1}=\frac{n(n+1)(n+3)(\gamma_{i}+\gamma_{n+1-i})}{\gamma_{n+1}}$$
%$$f_{2}= 2(n+1)(n+3)(1-3i(n+1-i))$$
%so, \begin{equation*}
%d^{-1}_{i, n}+d^{-1}_{i, 1} = \frac{\gamma_{1}(\gamma_{i}+\gamma_{n+1-i})}{36\gamma_{n+1}} - \frac{1}{36(n+2)}\left(\frac{n(\gamma_{i}+\gamma_{n+1-i})}{\gamma_{n+1}}+2(1-3i(n+1-i))\right)
%\end{equation*}
%2)
%\begin{equation*}
%d^{-1}_{i, n-1}+d^{-1}_{i, 2} = \frac{\gamma_{2}(\gamma_{i}+\gamma_{n+1-i})}{36\gamma_{n+1}} - \frac{1}{36(n+1)(n+2)(n+3)}\left(f_{1}+f_{2}\right)
%\end{equation*}
%where
%$$f_{1}= \frac{2(n-1)(n+1)(n+3)(\gamma_{i}+\gamma_{n+1-i})}{\gamma_{n+1}}$$
%$$f_{2}=6(n+1)(n+3)(n+3-3i(n+1-i))$$
%so, 
\theta_2 &=  \frac{\gamma_{2}(\gamma_{i}+\gamma_{n+1-i})}{36\gamma_{n+1}} - \frac{1}{36(n+2)}\left(\frac{2(n-1)(\gamma_{i}+\gamma_{n+1-i})}{\gamma_{n+1}}+6(n+3-3i(n+1-i))\right). \notag
\end{align}
Hence,
\begin{equation*}
g(i) = \frac{(6n+10)(\gamma_{n+1}-\gamma_{i}-\gamma_{n+1-i})}{36(n+2)\gamma_{n+1}}+\frac{i(n+1-i)}{6(n+2)}.
\end{equation*}
Note that
\begin{equation*}
\gamma'_{i} = \left(\frac{r^{i}_{1}-r^{i}_{2}}{2\sqrt{15}}\right)' = \frac{r^{i}_{1}\ln r_{1}-r^{i}_{2}\ln r_{2}}{2\sqrt{15}} = \frac{\alpha_{i} \ln r_{1} }{\sqrt{15}}
\end{equation*}
and $\ln r_{1}+\ln r_{2} =0$; Thus, $r_{1}r_{2}=1$.

Consider $i \in [1,n] \subset \mathbb{R}$. Differentiating $g(i)$ with respect to $i$ and solving the equation, we get
\begin{equation*}
g'(i) = \frac{(6n+10)\ln r_{1}}{36(n+2)\gamma_{n+1}\sqrt{15}}(\alpha_{n+1-i}-\alpha_{i})+\frac{n+1-2i}{6(n+2)} = 0.
\end{equation*}
Due to monotonicity of $g'(i)$, there is only solution of the above equation,  given by $i = \frac{n+1}{2}$,  This critical point is also the maximum point of the $g(i)$. Thus, 
\begin{equation*}
\pi_3 = \max_{1 \leq i\leq n} g(i) = g\left(\frac{n+1}{2}\right),
\end{equation*}
where
\begin{equation}
g\left(\frac{n+1}{2}\right) = \frac{6n+10}{36(n+2)} - \frac{2(6n+10)\gamma_{\frac{n+1}{2}}}{36(n+2)\gamma_{n+1}}+\frac{(n+1)^{2}}{24(n+2)} \leq \frac{3n^{2}+18n+23}{72(n+2)} \le \frac{3(n+2)(n+4)}{72(n+2)} = \frac{n+4}{24}. \notag
\end{equation}
%\subsection{Computation the last row sum}

For $\pi_2$, we first note that
\begin{equation*}
d^{-1}_{n, j} = \frac{c^{-1}_{n, j}}{36} - \frac{1}{36(n+1)(n+2)(n+3)}\left(f_{1}+f_{2}\right)
\end{equation*}
where
\begin{align*}
f_{1} &= \frac{(n+1-j)j(n+1)}{\gamma_{n+1}}(\gamma_{n+1-i}(n+2-j) +\gamma_{i}(j+1)), \\
f_{2} &= j\left(j+1\right)(2(j-1)i^3-3i(n+1)(i^2+ij+j+(n+2)(n+1-2i))\\
&+(n+1)\left((n+1)(n+2)(j-1)+j(2n+3)\right)).
\end{align*}
Therefore,
\begin{equation*}
\pi_2=\sum^{n}_{j=1}d^{-1}_{n, j} = \frac{\gamma_{1}}{36\gamma_{n+1}} \left(\sum^{n}_{j=1}\gamma_{j}\right)  - \frac{1}{36(n+1)(n+2)(n+3)}\left( \sum^{n}_{j=1}f_{1}  +  \sum^{n}_{j=1}f_{2} \right).
\end{equation*}
Using Lemma \ref{lem:2},
\begin{align*}
\sum^{n}_{j=1}\gamma_{j} &= \frac{1}{6}\left(\gamma_{n+1}-\gamma_{n}-1\right). \\
\sum^{n}_{j=1}f_{1} &= \frac{n+1}{\gamma_{n+1}}  \left( \gamma_{1}\sum^{n}_{j=1}(j^{3}-2j^{2}n-3j^{2}+jn^{2}+3jn+2j)  +   \gamma_{n} \sum^{n}_{j=1}(-j^{3}+j^{2}n+jn+j)  \right) \\
&= \frac{n(n+1)^{2}(n+2)(n+3)(\gamma_{1}+\gamma_{n})}{12\gamma_{n+1}}, \\
\sum^{n}_{j=1}f_{2} &=\sum^{n}_{j=1}(7j^{3}n+5j^{3}-7j^{2}n^{2}-4j^{2}n+3j^{2}-7jn^{2}-11jn-2j) = -\frac{n(n+1)(n+2)(n+3)(7n+1)}{12}.
\end{align*}
Substitution of all terms gives
\begin{align*}
\pi_2 &= \frac{\gamma_{1}(\gamma_{n+1}-\gamma_{n}-1)}{216\gamma_{n+1}} - \frac{1}{36}\left(\frac{n(n+1)(\gamma_{1}+\gamma_{n})}{12\gamma_{n+1}} - \frac{n(7n+1)}{12}\right) \\
&= \frac{1}{216}+\frac{n(7n+1)}{432}- \frac{(n^{2}+n+2)(\gamma_{n}+1)}{432\gamma_{n+1}}.
\end{align*}
From Lemma~\ref{lem:1}(ii), $\gamma_{n+1}/8 \le \gamma_n \le \gamma_n +1$. Thus,
\begin{equation}
\pi_2 \leq \frac{1}{216}+\frac{n(7n+1)}{432}- \frac{(n^{2}+n+2)}{432\times 8} = \frac{55n^{2}+7n+14}{3456} \leq \frac{56(n+1)^2}{3456} \le \frac{(n+1)^2}{432}. \notag
\end{equation}

%\subsection{Estimation of the infinity norm of inverse matrix $D^{-1}$}
Lastly,
\begin{equation*}
\pi_1=\max_{1 \leq i \leq n} \sum^{n}_{j=1}|d^{-1}_{i,j}| = \|D^{-1}_n\|_{\infty}  \leq \|C^{-1}_n\|_{\infty} \|B^{-1}_n\|_{\infty}.
\end{equation*}
By using Theorem 4 of \cite{hoskins1972some}
\begin{equation*}
\|B^{-1}_n\|_{\infty} \leq \frac{(n+1)^{2}(n+3)^{2}}{384}
\end{equation*}
and Lemma~\ref{lem:invCn},
\begin{equation}
\pi_1=\|D^{-1}_n\|_{\infty} \leq \|C^{-1}_n\|_{\infty} \|B^{-1}_n\|_{\infty} \leq \frac{(n+1)^{2}(n+3)^{2}}{2304}. \notag
\end{equation}

Summing up $\pi_1$, $\pi_2$, and $\pi_3$ and using the fact that $m_{11}+m_{12} > 1$ leads to the statement in the theorem.
\end{proof}

\section{The near Toeplitz case} \label{sec:Toeplitznear}

We now consider the seven-diagonal near Toeplitz matrix ~\eqref{form:An} with $a_0 = 68$ and $a_1 = -40$. We shall denote this matrix by $\wtA_n$ and use ``$\widetilde{\, \,\, \cdot \,\,\,}$'' to indicate perturbed matrices relevant to $\wtA_n$. It can be shown that $\wtA$ admits rank-2 decomposition\footnote{In fact, there exist two a rank-2 decomposition of $\wtA_n$. We choose this version as it shares many same components in the decomposition as in the Toeplitz case.}
\begin{align}
   \wtA_n = \wtB_n C_n + 2 UV^T,
\end{align}
where $C_n$, $U$, and $V$ are given in~\eqref{matCn} and~\eqref{banded}, and
\begin{align}
\wtB_n=
\begin{pmatrix}
7 & -4 & 1 & 0 &  \cdots  & \cdots  & 0 \\
-4 & 6 & -4 & 1 & \ddots & \ddots   & \vdots  \\
1 & -4 & 6 & -4 & 1 & \ddots  &  \vdots  \\
0 & \ddots  & \ddots  & \ddots & \ddots & \ddots & 0 \\
 \vdots  & \ddots  & 1 & -4 & 6 & -4 & 1 \\
 \vdots  & \ddots  & \ddots  & 1 & -4 & 6 & -4 \\
0 & \cdots & \cdots & 0 & 1 & -4 & 7
\end{pmatrix}
_{n\times n}.
\end{align}
The inverse of $\wtB_n$ is discussed in~\cite{kurmanbek2021pentadiag} and is given entry-wise by the explicit formula, for $i \ge j$,
\begin{eqnarray}
   \wtb_{i,j}^{-1} = \beta \left[ \epsilon + (j^2 -1)(2 \delta^2 + 1) \right], \label{form:btilde}
\end{eqnarray}
with
\begin{eqnarray}
  \delta &=& n + 1 - i, \notag \\
 \beta &=& \frac{\delta j}{6(n+1)(n^2 + 2n + 3)}, \notag \\
  \epsilon &=& 3[1+\delta (n+1)][1 + (i-j)j], \notag
\end{eqnarray}
and $\wtb^{-1}_{j,i} = \wtb^{-1}_{i,j}$ for $i < j$. Furthermore,
\begin{lemma} \label{lem:btspd}
$\wtB_{n}$ is positive definite.
\end{lemma}
\begin{proof}
For any nonzero vector $\bsx = (x_1 \dots x_n)^T\in \mathbb{R}^n$, $\bsx^T \wtB \bsx = \bsx^T B \bsx + x_1^2 + x_n^2 > 0$.
%\begin{align*}
%\bsx^{T}\wtB_{n}\bsx &= 7x_{1}^2 + 6(x_{2}^2 + \dots + x_{n-1}^2) +7x_{n}^2 - 8(x_{1}x_{2}+  \dots +x_{n-1}x_{n}) +2(x_{1}x_{3}+  \dots +x_{n-2}x_{n}) \\
%& = 2x_{1}^2+(2x_{1}-x_{2})^2 + (x_{1}-2x_{2}+x_{3})^{2}+ (x_{2}-2x_{3}+x_{4})^2+...+(x_{n-2}-2x_{n-1}+x_{n})^2 \\
%&+(2x_{n}-x_{n-1})^2+2x_{n}^2 \ge 0,
%\end{align*}
%with equality holding only when $\bsx=0$.
\end{proof}

\subsection{Exact inverse formula}

Let $\wtD_n = \wtB_n C_n = [\wtd_{i,j}]$. The general inverse formula of $\wtD_n$ is given by~\eqref{eq:dinv}, with $\wtd$ and $\wtb$ replacing $d$ and $b$, respectively. If $\wtA$ is invertible, its inverse can be expressed as
\begin{eqnarray}
   \wtA_n^{-1} = \wtD_n^{-1} - 2 \wtD_n^{-1} U \wtM^{-1} V^T \wtD_n^{-1}, 
\end{eqnarray}
where $\wtM  = [\wtm_{i,j}]  = I_2 + 2V^T \wtD^{-1} U\in \mathbb{R}^{2 \times 2}$ with
\begin{align}
    \wtm_{11} &= 1 + 8 \wtd^{-1}_{1,1} -  2 \wtd^{-1}_{1,2} = 1 - 2 \wtd^{-1}_{n,n-1} + 8 \wtd^{-1}_{n,n} = \wtm_{22}, \label{eq:mt11}\\
   \wtm_{12} &=  -2 \wtd^{-1}_{1,n-1} + 8 \wtd^{-1}_{1,n} = 8 \wtd^{-1}_{n,1} - 2 \wtd^{-1}_{n,2} = \wtm_{21}. \label{eq:mt12}
\end{align}
%due to centrosymmetric $\wtB_n$ and $C_n$.
An explicit form for $\wtm_{11}$ and $\wtm_{12}$ can be obtained 
via direct calculations using the formulas~\eqref{eq:dinv} and ~\eqref{form:btilde} and Lemma~\ref{lem:2}, and is given by
\begin{align}
    \wtm_{11} &= 1 + \frac{3(n^3 + 3n^2 + n + 1)\gamma_{n+1} - 3(n^3 + 3n^2 + 4n + 2) \gamma_n - 3(n+1)}{9 \gamma_{n+1} (n+1)(n^2 + 2n + 3)}, \notag \\
    \wtm_{12} &= \frac{6(2n + 1) \gamma_{n+1} - 3 (n+1) \gamma_n - 3(n+1) ((n+1)^2 + 1 )}{9 \gamma_{n+1} (n+1)(n^2 + 2n + 3)}. \notag 
\end{align}

\begin{lemma} \label{lem:mtspd}
$\wtM$ is a positive, diagonally dominant matrix. Moreover, $\det(\wtM) >0$.
\end{lemma}
\begin{proof}
Writing $\wtm_{11} = 1 + \tau_{11}/9 \gamma_{n+1} (n+1)(n^2 + 2n + 3)$ with $\tau_{11} = (3n^3 + 9n^2 + 3n + 3)\gamma_{n+1} - 3(n^3 + 3n^2 + 4n + 2) \gamma_n - 3(n+1)$, we have, for $n \ge 1$,
\begin{align}
  \tau_{11} &= %3 (n^3 + 3n^2 + n + 1)\gamma_{n+1} - 3(n^3 + 3n^2 + 4n + 2) \gamma_n - 3(n+1) &\notag \\
3[(n^3 + 3n^2 + n + 1) (\gamma_{n+1} - \gamma_n) - (3n + 1)\gamma _n - (n+1)] &\notag \\
         &\ge  3[(n^3 + 3n^2 + n + 1) (\gamma_{n+1} - \gamma_n) - 3(n+1)(\gamma_n +1)]  &\notag \\
         &=  3 \gamma_n [(n^3 + 3n^2 + n + 1) (\gamma_{n+1}/\gamma_n - 1) - 3(n+1)(1 + 1/\gamma_n)] & \notag \\
         &\ge  3 \gamma_n [(n^3 + 3n^2 + n + 1) (4 - 1) - 3(n+1)(1 + 1/\gamma_n)]  & \text{(Lemma 2(ii))} \notag \\
         &\ge 3 \gamma_n [3(n^3 + 3n^2 + n + 1)  - 6(n+1)] &  \notag \\
          &\ge  9 \gamma_n (n^3 + 3n^2 - 2n -1 ) > 0. & \notag
\end{align}
Thus, $\wtm_{11} > 1$.

Let $\tau_{12} = 6(2n + 1) \gamma_{n+1} - 3 (n+1) \gamma_n - 3(n+1) ((n+1)^2 + 1 )$, the numerator term of $\wtm_{1,2}$. By using Lemma~\ref{lem:1}(ii), we have $\tau_{12} = 6(2n + 1) \gamma_{n+1} - 3 (n+1) \gamma_n - 3(n+1) ((n+1)^2 + 1 ) \geq \frac{45n+21}{4}\gamma_{n+1} - 3(n+1)((n+1)^2+1) > 3(n+1)(\gamma_{n+1}-(n+1)^2-1)>0$, where the inequality $\gamma_{n+1}-(n+1)^2-1>0$ can be proved by induction. Thus, $\wtm_{12} > 0$, which shows the positivity of $\wtM$.
%Using the \text{Lemma 2 (ii)} for $\wtm_{12}$, the numerator term  $\tau_{12} = 6(2n + 1) \gamma_{n+1} - 3 (n+1) \gamma_n - 3(n+1) ((n+1)^2 + 1 ) \geq \frac{45n+21}{4}\gamma_{n+1} - 3(n+1)((n+1)^2+1) > 3(n+1)(\gamma_{n+1}-(n+1)^2-1)>0$. The inequality $\gamma_{n+1}-(n+1)^2-1>0$ can be proved by math induction. Thus $\wtm_{12} > 0$, which shows the positivity of $\wtM$.

Moreover,
\begin{align}
 \tau_{11} - \tau_{12} &= (3n^3 + 9n^2 - 9n - 3)\gamma_{n+1} - 3(n^3 + 3n^2 + 3n + 1) \gamma_n + 3(n+1) ((n+1)^2) \notag \\
  &\ge 3(n^3 + 3n^2 - 3n - 1)\gamma_{n+1} - \frac{3(n+ 1)^3\gamma_{n+1}}{4} + 3(n+1)^3 \notag \\
 &=\frac{3(3n^3 + 9n^2 - 15n - 5)\gamma_{n+1}}{4} + 3(n+1)^3 \notag \\
&\ge  3\gamma_{n+1}(n^2-3n)>0 \notag
\end{align}
for $n \ge 3$.
%Here, $3(n^3 + 3n^2 - 3n - 1) > 18(n+1)$ for $n \ge 3$.
%Furthermore,
%\begin{eqnarray}
%  \gamma_{n+1} - \gamma_n - 1 - (\gamma_n + 2) = \gamma_{n+1} - 2\gamma_n - 3 = \gamma_n \left( \frac{\gamma_{n+1}}{\gamma_n} - 2 - \frac{3}{\gamma_n} \right) \ge \gamma_n \left( \frac{\gamma_{n+1}}{\gamma_n} - 3 \right) \ge \gamma_n \left( 4 - 3 \right) = \gamma_n > 0.
%\end{eqnarray}
%Thus, $ \gamma_{n+1} - \gamma_n - 1 > (\gamma_n + 2)$. Putting together, we $3(n^3 + 3n^2 - 3n - 1)(\gamma_{n+1} - \gamma_n - 1) > 18(n+1)(\gamma_n + 2)$. 
Hence, $\tau_{11} > \tau_{12}$ and consequently, $\wtm_{11} > \wtm_{12}$.
\end{proof}

\begin{theorem} \label{Antilde_spd}
$\wtA_n$ is positive definite.
\end{theorem}
\begin{proof}
%For any nonzero vector $\bsx = (x_1 \dots x_n)^T \in \mathbb{R}^n$, $\bsx^T \wtA_n \bsx = \bsx^T A_n \bsx + x_1^2 + x_n^2 > 0$, due to Theorem~\ref{Ainvspd}.
%The proof is similar to the proof of Theorem~\ref{Ainvspd}, which is based on Sylvester's criterion. Let $\wtA_{k,k}$ be the the upper left $k \times k$ submatrix of $\wtA_n$. For $k = 1,\dots,6$, numerical calculation shows $\det(\wtA_{k,k}) >0$. For $k = n$, using the generalized matrix determinant lemma, $\det(\wtA_{n,n}) = \det(\wtA_n) = \det(\wtM) \det(\wtB_n) \det(\wtC_n) > 0$, due to Lemmas~\ref{lem:invCn},~\ref{lem:btspd}, and~\ref{lem:mtspd}. For $k = 7,\dots, n-1$, ???
Let us denote an upper-left $k \times k$ matrix of $\wtA_{}$ as $\wtA_{k,k}$. By Silvester's criterion, we need to show $\det{\wtA_{k,k}}>0$ for $k \in \{1, 2, ..., n\}$. For $k \in {1,\dots, 6}$, $\det (A_{k,k}) > 0$ by numerical calculation. For $k = n$, $\det{\wtA_{n,n}}=\det{\wtA_{n}} = \det{\wtM_{}}\det{\wtB_{n}}\det{\wtC_{n}}>0$, due to Lemmas~\ref{lem:invCn},~\ref{lem:btspd}, and~\ref{lem:mtspd}.

For $k \in \{7,\dots,n-1\}$, $\wtA_{k,k}$ is a seven-diagonal nearly Toeplitz matrix ~\eqref{form:An}, with perturbed top-left corner. For any nonzero vector $\bsx \in \mathbb{R}^k$, then 
\begin{align}
\bsx^{T}\wtA_{k,k}\bsx &= 68x_{1}^{2}+56(x_{2}^{2}+\dots+x_{k}^{2}) - 80x_{1}x_{2}-78(x_{2}x_{3}+\dots+x_{k-1}x_{k}) \notag \\ &+24(x_{1}x_{3}+\dots+x_{k-2}x_{k})-2(x_{1}x_{4}+\dots+x_{k-3}x_{k}). \notag
\end{align}
Consider $$S = \sum_{i=1}^{k-3}(ax_{i}-bx_{i+1}+cx_{i+2}-dx_{i+3})^2$$ where $a = \sqrt{4-\sqrt{15}}$, $b = (6+\sqrt{15})a$,  $c=(9+2\sqrt{15})a$, and $d=(4+\sqrt{15})a$, with properties 
$$
\begin{cases}
	a^{2}+b^{2}+c^{2}+d^{2} = 56, \\
	ab+bc+cd=39, \\
	ac+bd = 12,\\
	ad = 1. \\
	\end{cases}
	$$
Then
\begin{align*}
   \bsx^{T}\wtA_{k,k}\bsx &= S + (68-a^2)x_{1}^2+(c^2+d^2)x_{2}^2+d^2x_{3}^2+a^2x_{k-2}^2+(a^2+b^2)x_{k-1}^2+(a^2+b^2+c^2)x_k^2 \notag \\
   &-(80-2ab)x_1x_2-2cdx_{2}x_{3}-2abx_{k-2}x_{k-1}-(2ab+2bc)x_{k-1}x_{k}+2bdx_1x_3+2acx_{k-2}x_{k} \notag \\
   &= S + (68-a^2-b^2)x_1^2+d^2x_{2}^2+(bx_1-cx_2+dx_3)^2-(80-2ab-2bc)x_1x_2 \notag \\
   &+(ax_{k-2}-bx_{k-1}+cx_{k})^2+a^2x_{k-1}^2-2abx_{k-1}x_{k}+(a^2+b^2)x_{k}^2 \notag \\
   &= S + (12+c^2+d^2)x_1^2+d^2x_2^2+(bx_1-cx_2+dx_3)^2 - (2+2cd)x_1x_2+(ax_{k-2}-bx_{k-1}+cx_{k})^2 \notag \\
   &+ (ax_{k-1}-bx_{k})^2+a^2x_k^2 \notag \\
   &= S+(\sqrt{12+c^2+d^2}x_1-\frac{1+cd}{\sqrt{12+c^2+d^2}}x_2)^2+\Big(d^2 - \frac{(1+cd)^2}{12+c^2+d^2}\Big)x_2^2+(bx_1-cx_2+dx_3)^2 \notag \\
   &+(ax_{k-2}-bx_{k-1}+cx_{k})^2+(ax_{k-1}-bx_{k})^2+a^2x_k^2 \geq 0,
\end{align*}
with equality holding only when $\bsx =0$. (Note that $d^2 - \frac{(1+cd)^2}{12+c^2+d^2} \approx 2.20$ $>0$)
%$x^{T}\wtA_{k,k}x = \\
%= S + (68-a^2)x_{1}^2+(c^2+d^2)x_{2}^2+d^2x_{3}^2+a^2x_{n-2}^2+(a^2+b^2)x_{n-1}^2+(a^2+b^2+c^2)x_n^2 - (80-2ab)x_1x_2-2cdx_{2}x_{3}-2abx_{n-2}x_{n-1}-(2ab+2bc)x_{n-1}x_{n}+2bdx_1x_3+2acx_{n-2}x_{n} = \\ = S + (68-a^2-b^2)x_1^2+d^2x_{2}^2+(bx_1-cx_2+dx_3)^2-(80-2ab-2bc)x_1x_2+(ax_{n-2}-bx_{n-1}+cx_{n})^2+a^2x_{n-1}^2-2abx_{n-1}x_{n}+(a^2+b^2)x_{n}^2  = \\
%= S + (12+c^2+d^2)x_1^2+d^2x_2^2+(bx_1-cx_2+dx_3)^2 - (2+2cd)x_1x_2+(ax_{n-2}-bx_{n-1}+cx_{n})^2+(ax_{n-1}-bx_{n})^2+a^2x_n^2 = \\
%= S+(\sqrt{12+c^2+d^2}x_1-\frac{1+cd}{\sqrt{12+c^2+d^2}}x_2)^2+\Big(d^2 - \frac{(1+cd)^2}{12+c^2+d^2}\Big)x_2^2+(bx_1-cx_2+dx_3)^2+(ax_{n-2}-bx_{n-1}+cx_{n})^2+(ax_{n-1}-bx_{n})^2+a^2x_n^2 \geq 0$. with equality holding only when $x = 0$. \\
%Note $d^2 - \frac{(1+cd)^2}{12+c^2+d^2} \approx 7.55$\\
%Thus, $\wtA_{k,k}$ is positive-definite, and hence $\det{\wtA_{k,k}}>0$.
\end{proof}

\subsection{Bounds of norms of inverse matrix}

In this section, we shall derive an upper bound for norms of $\wtA^{-1}_n$. As we did for $A_n^{-1}$, the derivation will be given only for $p=\infty$.

Positive definiteness of $\wtA_n$ implies that $\wtA_n^{-1}$ is positive. Consequently, 
\begin{eqnarray}
  \sum_{j=1}^n |\wta^{-1}_{i,j}| = \sum_{j=1}^n \wta^{-1}_{i,j}= \sum_{j=1}^n \wtd^{-1}_{i,j} - \frac{2}{\wtm_{11} + \wtm_{12}} \left( \sum_{j=1}^n \wtd^{-1}_{n,j} \right) \left( 4(\wtd_{i,n}+ \wtd^{-1}_{i,1}) - (\wtd^{-1}_{i,n-1} + \wtd^{-1}_{i,2}) \right). \notag
\end{eqnarray}
The following inequality can be derived using the above expression:
\begin{align}
  \| \wtA^{-1} \|_{\infty} &= \max_i \sum_{j = 1}^n | \wta^{-1}_{i.j}| = \max_i \left\{ \sum_{j=1}^n \wtd^{-1}_{i,j} - \frac{2}{\wtm_{11} + \wtm_{12}} \left( \sum_{j=1}^n \wtd^{-1}_{n,j} \right) \left( 4(\wtd_{i,n}+ \wtd^{-1}_{i,1}) - (\wtd^{-1}_{i,n-1} + \wtd^{-1}_{i,2}) \right) \right\} \notag \\
&\le \underbrace{\max_i \sum_{j=1}^n \wtd_{i,j}^{-1}}_{\wtpi_1} + \frac{2}{\wtm_{11} + \wtm_{12}} \underbrace{\left( \sum_{j=1}^n \wtd^{-1}_{n,j} \right)}_{\wtpi_2}  \underbrace{\max_i \wtg(i)}_{\wtpi_3}, \label{eq:invAtd}
\end{align}
where
\begin{eqnarray}
   \wtg(i) = 4 \underbrace{(\wtd_{i,n}^{-1}+ \wtd^{-1}_{i,1})}_{\wttet_1} - \underbrace{(\wtd^{-1}_{i,n-1} + \wtd^{-1}_{i,2})}_{\wttet_2}.
\end{eqnarray}
%An upper bound for each term in~\eqref{} $\wtP_1$ and~\eqref{} $\wtP_2$ is derived in the subsequent sections.

% \subsubsection{The $\wtP_1$ term}
%For $\wtP_1$,
With $\|B_n^{-1}\|_{\infty} \le (n+1)^2((n+1)^2+8)/384$ (see ~\cite{dow2002explicit}) and Lemma~\ref{lem:invCn}, we have
\begin{align}
   \wtpi_1 = \max_i \sum_{j=1}^n \wtd^{-1}_{i,j} &= \| \wtD^{-1} \|_{\infty} \le \| \wtB^{-1} \|_{\infty} \| C^{-1} \|_{\infty} \le \frac{(n+1)^2 ((n+1)^2 + 8)}{2304}. \notag
\end{align}

% \subsubsection{The $\wtP_2$ term} 

Next,
\begin{eqnarray}
   \wtd^{-1}_{n,j} =  \frac{\gamma_1}{\gamma_{n+1}} \left(  \sum_{k = 1}^{j} \gamma_k \wtb^{-1}_{k,j} + \sum_{k = j+1}^n \gamma_k \wtb^{-1}_{k,j} \right) = \frac{\gamma_1}{\gamma_{n+1}} \left(  \sum_{k = 1}^{j} \gamma_k \wtb^{-1}_{j,k} + \sum_{k = j+1}^n \gamma_k \wtb^{-1}_{k,j} \right). \notag
\end{eqnarray}
By using the explicit formula for $\wtb^{-1}_{i.j}$, $i \ge j$ and Lemma~\ref{lem:2}, after tedious calculation we get
\begin{align}
   \wtd^{-1}_{n,j} &=\mu (\nu_3 j^3 + \nu_2 j^2 + \nu_1 j + \nu_0 \gamma_j),
\end{align}
where
\begin{align*}
  \mu &= \frac{\gamma_1} {36 \gamma_{n+1} (n+1)(n^2 + 2n + 3)}, \\
  \nu_0 &= n^3 + 3n^2 + 5n + 3, \\
  \nu_1 &= 2(2n+1)\gamma_{n+1} - (n+1)\gamma_n - (n^3 + 3n^2 + 4n + 2), \\
  \nu_2 &= (4n^2 + 5n -3)\gamma_{n+1} - n(n+2)\gamma_n  + 2n^2 + 4n + 3, \\
  \nu_3 &= -2(2n+1) \gamma_{n+1} + (n+1)\gamma_n - (n+1). 
\end{align*}
Therefore,
\begin{align}
   \wtpi_2 &= \sum_{j=1}^n \wtd^{-1}_{n,j} = \kappa \left(\nu_3 \sum_{j=1}^n j^3 + \nu_2 \sum_{j=1}^n j^2 + \nu_1 \sum_{j=1}^n j + \nu_0 \sum_{j=1}^n \gamma_j \right), \notag \\
   &= \kappa \left(  \nu_3 \frac{n^2(n+1)^2}{4} + \nu_2 \frac{n(n+1)(2n+1)}{6} + \nu_1 \frac{n(n+1)}{2} + \nu_0 \frac{\gamma_{n+1} - \gamma_n -1}{6} \right) \notag \\
  &= \frac{\gamma_1}{36 \gamma_{n+1}} \left( \frac{1}{6}(2n^2 + n +1)(\gamma_{n+1}-1) + \frac{n^2}{2(n^2 + 2n + 3)} - \frac{1}{12}(n^2+2n + 2)\gamma_n \right) \notag \\
&\le \frac{\gamma_1}{36 \gamma_{n+1}} \left( \frac{1}{6}(2n^2 + n +1) \gamma_{n+1} + \frac{n^2}{2(n^2 + 2n + 3)} - \frac{1}{12}(n^2+2n + 2)\gamma_n \right). \notag
\end{align}
With Lemma~\ref{lem:1}(ii) and $\gamma_{n+1} \ge 4 \gamma_n \ge 4^2 \gamma_{n-1} \ge 4^{n+1}$, for $n \ge 2$,
\begin{align}
  \wtpi_2 &\le \frac{1}{36}  \left( \frac{1}{6}(2n^2 + n +1)  + \frac{n^2}{2\gamma_{n+1}(n^2 + 2n + 3)} - \frac{1}{96}(n^2+2n + 2) \right) \notag \\
   &= \frac{1}{36}\left( \frac{31n^2 + 14n + 14}{96} + \frac{n^2}{2\gamma_{n+1}(n^2 + 2n + 3)} \right) \notag \\
   &\le \frac{1}{36} \left( \frac{31n^2 + 14n + 14}{96} + \frac{1}{2 \cdot 4^3} \right) \notag \\
%   &= \frac{1}{36} \left( \frac{31n^2 + 14n + 14}{96} + \frac{1}{128} \right) \notag \\
%   &\le \frac{1}{36} \left( \frac{31n^2 + 14n + 14}{96} + \frac{1}{96} \right) \notag \\
   &\le \frac{1}{3456}(31n^2 + 14n + 15). \notag
\end{align}

%\subsubsection{The $\wtK$ term}

We now construct an estimate for $\wtpi_3$. Using~\eqref{eq:dinv}, we have
\begin{equation}
  \begin{aligned}
     \wtd_{i,1}^{-1} &= \frac{\gamma_{n+1-i}}{\gamma_{n+1}} \sum_{k=1}^i \gamma_k \wtb_{k,1}^{-1} +\frac{\gamma_i}{\gamma_{n+1}} \sum_{k=1}^{n-i} \gamma_k \wtb_{n+1-k,1}^{-1},\\
     \wtd_{i,n}^{-1} &= \frac{\gamma_i}{\gamma_{n+1}} \sum_{k=1}^{n+1-i} \gamma_k \wtb^{-1}_{k,1} + \frac{\gamma_{n+1-i}}{\gamma_{n+1}} \sum_{k=1}^{i-1} \gamma_k \wtb_{n+1-k,1}^{-1}, \\
     \wtd_{i,2}^{-1} &= \begin{cases} 
     \displaystyle   \frac{\gamma_1}{\gamma_{n+1}} \sum_{k=1}^n \gamma_k \wtb^{-1}_{k,n-1}, &i = 1, \\
     \displaystyle  \frac{\gamma_{n+1-i}}{\gamma_{n+1}} \sum_{k=1}^i \gamma_k \wtb^{-1}_{k,2} + \frac{\gamma_i}{\gamma_{n+1}} \sum_{k=1}^{n-i} \gamma_k \wtb^{-1}_{n+1-k,2},& 2 \le i \le n,
                                   \end{cases} \notag \\
     \wtd_{i,n-1}^{-1} &= \begin{cases}
     \displaystyle \frac{\gamma_i}{\gamma_{n+1}} \sum_{k=1}^{n+1-i} \gamma_k \wtb^{-1}_{k,2} + \frac{\gamma_{n+1-i}}{\gamma_{n+1}} \sum_{k=1}^{i-1} \gamma_k \wtb^{-1}_{n+1-k,2},& 1 \le i \le n-1, \\    
    \displaystyle \frac{\gamma_1}{\gamma_{n+1}} \sum_{k=1}^n \gamma_k \wtb^{-1}_{k,n-1},& i = n.
    \end{cases}
  \end{aligned}
\end{equation}
Hence
\begin{align}
      \wttet_1  &= \frac{\gamma_i}{\gamma_{n+1}} \sum_{k=1}^{n-i} \gamma_k \mathcal{B}_1(k)  + \frac{\gamma_i \gamma_{n+1-i} }{\gamma_{n+1}} (\wtb^{-1}_{i,1} + \wtb^{-1}_{n+1-i,1} )+ \frac{\gamma_{n+1-i}}{\gamma_{n+1}} \sum_{k=1}^{i-1} \left\{ \gamma_k \mathcal{B}_1(k) \right\},  \notag \\
  \wttet_2  &= \frac{\gamma_i}{\gamma_{n+1}} \sum_{k=1}^{n-i} \left\{ \gamma_k \mathcal{B}_2(k) \right\} + \frac{\gamma_i \gamma_{n+1-i} }{\gamma_{n+1}} (\wtb^{-1}_{i,2} + \wtb^{-1}_{n+1-i,2} )+ \frac{\gamma_{n+1-i}}{\gamma_{n+1}} \sum_{k=1}^{i-1} \left\{ \gamma_k \mathcal{B}_2(k), \right\}. \notag
\end{align}
where, for $k = 1,\dots, n$,
\begin{align*}
  \mathcal{B}_1(k) &= \wtb^{-1}_{k,1} + \wtb^{-1}_{n+1-k,1} = \frac{-k^2 + (n+1)k}{2(n+1)}, \\
  \mathcal{B}_2(k) &= \wtb^{-1}_{k,2} + \wtb^{-1}_{n+1-k,2} = \frac{-2k^2 + 2(n+1)k - (n+1)}{n+1}, 
\end{align*}
after direct calculations of each $\wtb$ using~\eqref{form:btilde}. Furthermore, $4\mathcal{B}_1(k) - \mathcal{B}_2(k) = 1$.

We shall now use the above intermediate results and Lemmas~\ref{lem:2} to derive an expression for $\wtg$:
\begin{align}
\wtg(i) &= 4\wttet_1(i) - \wttet_2(i) \notag \\
         &= \frac{\gamma_i}{\gamma_{n+1}} \sum_{k=1}^{n-i} \left\{ \gamma_k (4 \mathcal{B}_1(k) - \mathcal{B}_2(k)) \right\} + \frac{\gamma_i \gamma_{n+1-i} }{\gamma_{n+1}} (4 \mathcal{B}_1(k) - \mathcal{B}_2(k)) + \frac{\gamma_{n+1-i}}{\gamma_{n+1}} \sum_{k=1}^{i-1} \left\{ \gamma_k (4 \mathcal{B}_1(k) - \mathcal{B}_2(k)) \right\} \notag \\
        &=  \frac{\gamma_i}{\gamma_{n+1}} \sum_{k=1}^{n-i} \gamma_k  + \frac{\gamma_i \gamma_{n+1-i} }{\gamma_{n+1}}  + \frac{\gamma_{n+1-i}}{\gamma_{n+1}} \sum_{k=1}^{i-1}  \gamma_k \notag \\
%       &=  \frac{\gamma_i}{6\gamma_{n+1}} (\gamma_{n+1-i} - \gamma_{n-i} - 1) +  \frac{\gamma_i \gamma_{n+1-i} }{\gamma_{n+1}} + \frac{\gamma_{n+1-i}}{6\gamma_{n+1}} (\gamma_{i} - \gamma_{i-1} - 1) \notag \\
      &= \frac{4}{3} \frac{\gamma_i \gamma_{n+1-i}}{\gamma_{n+1}} - \frac{1}{6} \frac{\gamma_i \gamma_{n-i}}{\gamma_{n+1}} - \frac{1}{6} \frac{\gamma_{i-1}\gamma_{n+1-i}}{\gamma_{n+1}} - \frac{\gamma_i + \gamma_{n+i-1}}{6 \gamma_{n+1}}.
\end{align}
Considering $i \in [1,n] \subset \mathbb{R}$ and with
\begin{align}
   &(\gamma_i \gamma_{n+1-i})' = \frac{1}{2\sqrt{15}} \gamma_{n+1-2i} \ln \left( \frac{r_1}{r_2} \right), \notag \\
   &(\gamma_i \gamma_{n-i})' = \frac{1}{2\sqrt{15}} \gamma_{n-2i} \ln \left( \frac{r_1}{r_2} \right), \notag \\
   &(\gamma_{i-1} \gamma_{n+1-i})' = \frac{1}{2\sqrt{15}} \gamma_{n-2i+2} \ln \left( \frac{r_1}{r_2} \right), \notag
\end{align}
we have
\begin{align}
  \wtg'(i) &= \frac{4}{3} \left( \frac{\gamma_i \gamma_{n+1-i}}{\gamma_{n+1}} \right)'- \frac{1}{6} \left( \frac{\gamma_i \gamma_{n-i}}{\gamma_{n+1}} \right)' - \frac{1}{6} \left( \frac{\gamma_{i-1}\gamma_{n+1-i}}{\gamma_{n+1}} \right)'- \left( \frac{\gamma_i + \gamma_{i+1}}{6 \gamma_{n+1}} \right)' \notag \\
%           &= \frac{1}{2\sqrt{15} \gamma_{n+1} } \left( \frac{4}{3}  \gamma_{n+1-2i} - \frac{1}{6} \gamma_{n-2i} - \frac{1}{6} \gamma_{n-2i+2} \right) \ln \left(\frac{r_1}{r_2} \right) - \frac{(r_1^i - r_1^{n+1-i})\ln r_1 + (r_2^{n+1-i} - r_2^i)\ln r_2}{12 \sqrt{15}\gamma_{n+1}} \notag \\
%          &= \frac{1}{2\sqrt{15} \gamma_{n+1} } \left( \frac{4}{3}  \gamma_{n+1-2i} - \frac{1}{6} (\gamma_{n-2i} + \gamma_{n-2i+2}) \right) - \frac{r_1^i(1-r_1^{n+1-2i})\ln r_1  -r_2^i (1-r_2^{n+1-2i}) \ln r_2}{12 \sqrt{15}\gamma_{n+1}} \notag \\
%          &= \frac{1}{2\sqrt{15} \gamma_{n+1} } \left( \frac{4}{3}  \gamma_{n+1-2i} - \frac{8}{6}\gamma_{n+1-2i}\right) - \frac{r_1^i(1-r_1^{n+1-2i})\ln r_1  -r_2^i(1-r_2^{n+1-2i}) \ln r_2}{12 \sqrt{15}\gamma_{n+1}} \notag \\
         &= -\frac{r_1^i(1-r_1^{n+1-2i})\ln r_1 - r_2^i(1-r_2^{n+1-2i}) \ln r_2}{12 \sqrt{15}\gamma_{n+1}} \notag \\
         &=  \frac{\ln r_1(\alpha_{n+1-i} - \alpha_{i})}{6\sqrt{15}\gamma_{n+1}}. \notag 
\end{align}
The critical point is $i = (n+1)/2$, which is also the maximum of $\wtg(i)$. Therefore,
\begin{align}
   \wtpi_3 &\le \wtK((n+1)/2) = \frac{4}{3} \frac{\gamma_{\frac{n+1}{2}}^2}{\gamma_{n+1}} - \frac{1}{3} \frac{\gamma_{\frac{n+1}{2}} \gamma_{\frac{n-1}{2}}}{\gamma_{n+1}} - \frac{1}{3} \frac{\gamma_{\frac{n+1}{2}}}{\gamma_{n+1}} \notag \\
&\le \frac{4}{3} \frac{\gamma_{\frac{n+1}{2}}^2}{\gamma_{n+1}} - \frac{1}{24} \frac{\gamma_{\frac{n+1}{2}}^2}{\gamma_{n+1}} =  \frac{31\gamma_{\frac{n+1}{2}}^2}{24\gamma_{n+1}}  =  \frac{31}{48\sqrt{15}} \frac{(r_1^{\frac{n+1}{2}} - r_2^{\frac{n+1}{2}})^2}{r_1^{n+1} - r_2^{n+1}} \notag \\
%&= \frac{31}{24} \frac{1}{2\sqrt{15}} \frac{(r_1^{\frac{n+1}{2}} - r_2^{\frac{n+1}{2}})^2}{(r_1^{\frac{n+1}{2}} + r_2^{\frac{n+1}{2}})(r_1^{\frac{n+1}{2}} - r_2^{\frac{n+1}{2}})} \notag \\
%   &= \frac{31}{24}\frac{1}{2\sqrt{15}} \frac{1 - (r_2/r_1)^{\frac{n+1}{2}}}{1+(r_2/r_1)^{\frac{n+1}{2}}} \notag \\
   &= \frac{31}{48\sqrt{15}} \frac{1- 1/r_1^{n+1}}{1 + 1/r^{n+1}} \le \frac{31}{48\sqrt{15}}. \notag
\end{align}

\begin{theorem}\label{thm:norminv}
For the matrix~\eqref{form:An}, with $a_0 = 68$ and $a_1 = -40$, the following inequality holds for $p \in \{1,2,\infty\}$:
$$
   \| \wtA^{-1}_n \|_p \le \frac{(n+1)^2((n+1)^2 + 14)}{2304}.
$$
\end{theorem}
\begin{proof}
Notice that
\begin{align}
   \wtm_{11} + \wtm_{12} &= 1 + \frac{3(n^3 + 3n^2 + 15n + 9)\gamma_{n+1} - 3(n^3 + 3n^2 + 5n + 3)\gamma_n + 3(n+1)(n^2+2n+3)}{9 \gamma_{n+1}(n+1)(n^2 + 2n + 3)} \notag \\
   &= 1 + \frac{1}{3} + \frac{10n+6}{3(n+1)(n^2+2n+3)}- \frac{1}{3} \frac{\gamma_n}{\gamma_{n+1}} + \frac{1}{3\gamma_{n+1}} \notag \\
   &\ge \frac{4}{3} - \frac{1}{3} \frac{\gamma_n}{\gamma_{n+1}}
   \ge  \frac{4}{3} - \frac{1}{3} \times \frac{1}{4} = \frac{5}{4}, \notag
\end{align}
after applying Lemma~\ref{lem:1}(ii). By using $\wtpi_1$, $\wtpi_2$, $\wtpi_3$ as given above, the bound~\eqref{eq:invAtd} reads, for $n \ge 1$,
\begin{align}
    \|\wtA^{-1}\|_{\infty} &\le  \frac{1}{2304}((n+1)^2 ((n+1)^2 + 8) + \frac{31(31n^2 + 14n +15)}{103680\sqrt{15}} \notag \\
%                                 &\le  \frac{1}{2304}((n+1)^2 ((n+1)^2 + 8) + \frac{31n^2 + 14n +15}{10033} \notag \\
            &\le  \frac{(n+1)^2((n+1)^2 + 8)}{2304} + \frac{6(n+1)^2}{2304} = \frac{(n+1)^2((n+1)^2 + 14)}{2304}. \notag
\end{align}
\end{proof}

\section{Numerical examples} \label{sec:numer}

We have computed the norms of exact inverses of $A_n$ and $\wtA$ for various size $n$ and the proposed upper bounds from Theorem~\ref{theo:1} and~\ref{thm:norminv}. The computational results are presented in Figure~\ref{fig:norminv1} various matrix size $n$, which suggests a good estimate provided by the theorem. 
\begin{figure}[H]
	\centering
	\begin{subfigure}[b]{0.4\textwidth}
    \centering
	\includegraphics[width=1.0\linewidth]{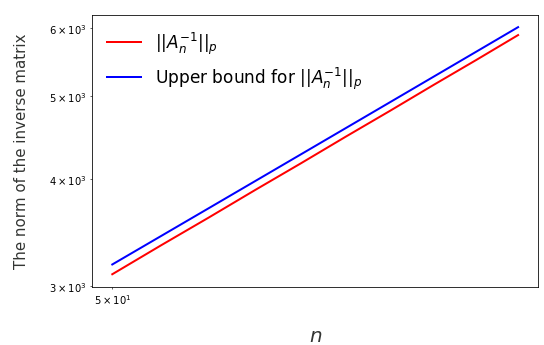}
	\caption{$A_n^{-1}$ (Toeplitz case) norm and the upper bound computations}
	\end{subfigure}
	\hfill
	\begin{subfigure}[b]{0.4\textwidth}
    \centering
      \includegraphics[width=1.0\linewidth]{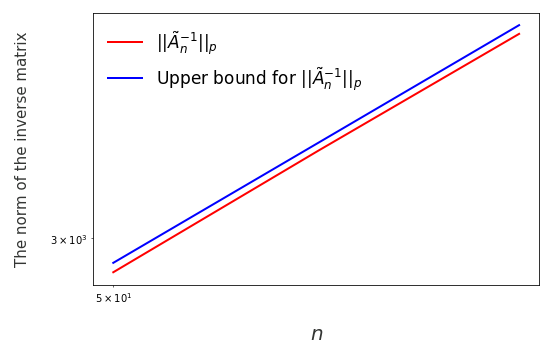}
      \caption{$\wtA_n^{-1}$ (nearly Toeplitz case) norm and upper bound computations}
      \end{subfigure}
	\caption{Evaluation of norm of inverse of matrices and bound in the log scale.}
	\label{fig:norminv1}
\end{figure}

\section{Conclusions} \label{sec:conclusion}

% we derive an explicit, non-recurrence formula for the inverse of seven-diagonalmatrix (1) witha0=68 anda1=−40. The inverse formula is then used to analyze some properties of the inverse matrix and construct an upper bound for the norms of the inverse, in terms of the matrix size n(which is linked to the mesh size h). Using the same line of analysis, results for theToeplitz version of (1) are also presented.
%The main goal of the current study was to 

In this paper, we derived the explicit formula of the inverse of seven-diagonal matrix and give upper bounds for its norms in terms of $n$. Findings have a great potential for other applications such as numerical analysis. Numerical verification was provided. The next stage of our research will be exploring more complicated matrices for the biharmonic problems with different boundary conditions and inverse properties of sevendiagonal near Toeplitz matrices with general perturbed corners. 

\section{Acknowledgement}

BK and YA wishes to acknowledge the research grant, No AP08052762, from the Ministry of Education and the Nazarbayev University Faculty Development Competitive Research Grant (NUFDCRG), Grant No 110119FD4502.

%\small
%\singlespacing
%\oneandonehalfspacing
\bibliographystyle{plain}%unsrt}plain  
\bibliography{reference}  %%% Remove comment to use the external .bib file (using bibtex).
\end{document}